\documentclass[oneside]{amsart}
\usepackage{amssymb}
\usepackage{amsmath}
\usepackage{amsfonts}
\usepackage{amsthm}
\usepackage{hyperref}
\usepackage{amscd}
\usepackage{color}
\usepackage{enumerate}
\usepackage{xypic}
\usepackage{tikz-cd}

\newcommand{\Aff}{\mathbb{A}}
\newcommand{\ZZ}{\mathbb{Z}}
\newcommand{\KK}{\mathbb{K}}

\newcommand{\QQ}{\mathbb{Q}}

\newcommand{\CO}{\mathcal{O}}

\newcommand{\T}{\mathcal{T}}

\newcommand{\F}{\mathcal{F}}
\newcommand{\U}{\mathcal{U}}
\newcommand{\V}{\mathcal{V}}
\newcommand{\W}{\mathcal{W}}
\newcommand{\sing}{\mathrm{sing}}
\newcommand{\ch}{\check{H}}
\newcommand{\fun}{\mathrm{fun}}

\DeclareMathOperator{\conv}{conv}

\DeclareMathOperator{\spec}{Spec}

\DeclareMathOperator{\sign}{sign}
\DeclareMathOperator{\Hom}{Hom}

\newtheorem{thm}{Theorem}[section]
\newtheorem{prop}[thm]{Proposition}
\newtheorem{lemma}[thm]{Lemma}
\newtheorem{cor}[thm]{Corollary}

\newtheorem{question}[thm]{Question}

\theoremstyle{definition}

\newtheorem{defn}[thm]{Definition}
\newtheorem{rem}[thm]{Remark}

\title[Smooth Toric Varieties: Obstructions and the Cup Product]{Deformations of Smooth Complete Toric Varieties: Obstructions and the Cup Product}

\author{Nathan Ilten}
\address{Department of Mathematics, Simon Fraser University,
8888 University Drive, Burnaby BC V5A1S6, Canada}
\address{Mathematical Institute of the Polish Academy of Sciences, 
Sniadeckich 8, 00-656 Warszawa}
\email{nilten@sfu.ca}

\author{Charles Turo}
\address{Department of Mathematics, Simon Fraser University,
8888 University Drive, Burnaby BC V5A1S6, Canada}
\email{cturo@sfu.ca}
\begin{document}

\begin{abstract}
Let $X$ be a complete $\QQ$-factorial toric variety. We explicitly describe the space $H^2(X,\T_X)$ and the cup product map $H^1(X,\T_X)\times H^1(X,\T_X)\to H^2(X,\T_X)$ in combinatorial terms. Using this, we give an example of a smooth projective toric threefold for which the cup product map does not vanish, showing that in general, smooth complete toric varieties may have obstructed deformations.
\end{abstract}
\maketitle
\section{Introduction}
\subsection{Background and Motivation}
Let $X$ be any variety over an algebraically closed field $\KK$ of characteristic not equal to two or three. The deformation theory of $X$ provides useful information on how $X$ might fit into a moduli space. The abstract theory guarantees that in good situations (e.g. $X$ complete or $X$ an isolated singularity) $X$ will possess a versal deformation, from which all deformations of $X$ can be induced. However in practice, the versal deformation of a given variety may be very difficult to describe in its entirety.
It is thus interesting to study classes of varieties for which one may more explicitly understand the deformation theory.

One special class of varieties whose deformation theory has been studied are \emph{toric varieties}. Deformations of such varieties have applications ranging from mirror symmetry \cite{mavlyutov1,coats} to K\"ahler-Einstein and extremal metrics \cite{tipler,kstab}. The deformation theory of \emph{affine} toric varieties has been described extensively by Altmann. Combinatorial formulas exist for the tangent and obstruction spaces $T_X^1$ and $T_X^2$ as well as a combinatorial description of the cup product map \cite{altmann1,altmann2}, see also recent work by Filip \cite{matej}. A combinatorial recipe may be used to construct deformations of $X$ over affine space \cite{altmann3}, and in some cases (e.g.~isolated Gorenstein singularities) there is an explicit combinatorial description of the entire versal deformation \cite{altmann4}.

In this paper, we will instead continue the program initiated by the first author in \cite{ilten1} of describing the deformation theory of \emph{smooth, complete} toric varieties.  
Let $X$ be a smooth complete toric variety corresponding to a fan $\Sigma$.
In loc.~cit.~Ilten gave a combinatorial description of the space 
$T_X^1=H^1(X,\T_X)$ of first order deformations:
\begin{equation}\label{eqn:t1}	
H^1(X,\T_X)=\bigoplus_{\rho\in\Sigma^{(1)}} \bigoplus_{\substack{u\in M\\ \rho(u) =-1}} \widetilde{H}^0(\Gamma_{\rho,u},\KK)
\end{equation}
where $\rho$ ranges over all rays of the fan $\Sigma$, $M$ is the character lattice of the torus of $X$, $\rho(u)\in\ZZ$ denotes the pairing between the primitive generator of $\rho$ and $u$, and $\Gamma_{\rho,u}$ is a certain graph, see \S\ref{sec:obs}.
Here, $\widetilde{H}$ denotes reduced cohomology.

Generalizing Altmann's construction in the affine case, Ilten and Vollmert gave a recipe for producing deformations of any toric variety $X$ over affine spaces from combinatorial data \cite{ilten2}, see also work by Mavlyutov \cite{mavlyutov2} and Petracci \cite{petracci}. In particular, when $X$ is smooth and complete, each connected component of a graph $\Gamma_{\rho,u}$ appearing in \eqref{eqn:t1} gives rise to a one-parameter deformation (over $\Aff^1$) lifting the corresponding first order deformation in $T_X^1$ \cite[Theorem 6.5]{ilten2}. In fact, for any character $u\in M$, one may use this construction to produce a deformation over $\Aff^m$ whose image in $T_X^1$ spans the entire degree $u$ piece. This is evidence that, despite in general having non-vanishing obstruction spaces, smooth complete toric varieties might have unobstructed deformations, similar to the situation of e.g.~Calabi-Yau varieties \cite{tian,todorov}. However, we will see below that this is not the case.

\subsection{Results}
Throughout, $X$ will be a complete $\QQ$-factorial toric variety corresponding to a fan $\Sigma$ with character lattice $M$. 
The description \eqref{eqn:t1} of $H^1(X,\T_X)$ in the case $X$ smooth also holds when $X$ is only $\QQ$-factorial, see \S\ref{sec:obs}.
There is also a straightforward generalization of \eqref{eqn:t1} for $H^2(X,\T_X)$: 
\begin{prop}[Proposition \ref{prop:obs2}]\label{prop:obs}
	The cohomology group $H^2(X,\T_X)$ may be decomposed as
	\begin{equation}\label{eqn:t2}
H^2(X,\T_X)=\bigoplus_{\rho\in\Sigma^{(1)}} \bigoplus_{\substack{u\in M\\ \rho(u) =-1}} H^1(K_{\rho,u},\KK)
\end{equation}
where each $K_{\rho,u}$ is a simplicial complex determined from $\Sigma$, see \S\ref{sec:obs}.
\end{prop}

Our main result is then to give a combinatorial description of the cup product map
\[
H^1(X,\T_X)\times H^1(X,\T_X)\to H^2(X,\T_X)
\]
using \eqref{eqn:t1} and \eqref{eqn:t2}.
When $X$ is smooth, 
$H^2(X,\T_X)$ is the obstruction space $T_X^2$, and the cup product may be used to obtain the quadratic terms in the obstruction equations for the versal deformation of $X$.
To describe the cup product, we will use \v{C}ech cohomology (with respect to a closed covering) to describe elements of the cohomology groups $\widetilde H^0(\Gamma_{\rho,u},\KK)$ and $H^1(K_{\rho,u}, \KK)$. 
The closed covering we consider will be indexed by maximal cones $\sigma\in\Sigma$; the corresponding closed sets will be the intersections of $\sigma$ with either $\Gamma_{\rho,u}$ or $K_{\rho,u}$.

\begin{thm}[Theorem \ref{thm:main2}]\label{thm:main}
	Fix $\rho,\rho'\in\Sigma^{(1)}$ and $u,u'\in M$ with $\rho(u)=\rho'(u')=-1$.
	\begin{enumerate}
		\item The image of
\[\widetilde{H}^0(\Gamma_{\rho,u},\KK)\times \widetilde{H}^0(\Gamma_{\rho',u'},\KK)\]
in $H^2(X,\T_X)$ under the cup product via \eqref{eqn:t1} is $0$ unless $\rho(u')=0$ or $\rho'(u)=0$.
\item Assume that $\rho(u')=0$, and let
	$f=(f_\sigma)$ and $f'=(f_\sigma')$ be \v{C}ech zero-cycles of $\Gamma_{u,\rho}$ and $\Gamma_{u',\rho'}$. 
	Then the cup product of $(\overline f,\overline f')$ is contained in $H^1(K_{\rho,u+u'})$ via \eqref{eqn:t2} and may be represented by the \v{C}ech one-cocycle $g=(g_{\sigma\tau})$ where
	\[
		g_{\sigma\tau}=\begin{cases}
			\frac{\rho'(u)}{2}(f_{\sigma}f_{\tau}'-f_{\tau}f_{\sigma}') & K_{\rho,u+u'}\cap\sigma\cap\tau\neq\emptyset,\\
			0& \textrm{otherwise}.
	\end{cases}
\]
A similar formula holds when $\rho'(u)=0$.
\end{enumerate}
\end{thm}

While this theorem gives an explicit description of the cup product on the combinatorial level, it is perhaps not always immediately obvious when the one-cocycle $(g_{\sigma\tau})$ is non-trivial. To remedy this, we proceed as follows.
Assume as in the second part of the theorem that $\rho(u')=0$.
Consider any simple cycle $\alpha$ in $K_{\rho,u+u'}$, and connected components $Z$ and $Z'$ of $\Gamma_{\rho,u}$ and $\Gamma_{\rho',u'}$. Then $H^1(\alpha,\KK)=\KK$ with canonical generator $\alpha_\fun$, and $Z$ and $Z'$ induce elements of $\widetilde H^0(\Gamma_{\rho,u},\KK)$ and $\widetilde H^0(\Gamma_{\rho',u'},\KK)$. The pullback of the cup product of these elements to $H^1(\alpha,\KK)$ is $(Z*_{\alpha}Z')\cdot \alpha_\fun$, where $Z*_{\alpha}Z'$ is determined from the intersection behaviour of $Z$ and $Z'$ along $\alpha$, see \S\ref{sec:cyclethm} and Theorem \ref{thm:cycle} for a precise statement. 

This leads to a straightforward method to determine when the cup product vanishes. In particular, we may easily use this to construct examples of smooth toric threefolds where the cup product does not vanish:

\begin{cor}[Corollary \ref{cor:ob2}]\label{cor:obs}
There exists a smooth complete toric threefold with obstructed deformations.
\end{cor}

\subsection{Murphy's Law and Future Directions}
Is this obstructedness result (Corollary \ref{cor:obs}) surprising? We would argue that although perhaps not surprising, it is far from obvious. On the one hand, Vakil has shown
\emph{Murphy's Law} for several classes of deformation problems, that is,
that arbitrarily bad singularities of finite type over $\ZZ$ can occur in the versal deformations \cite{vakil}. For example, this is true for smooth projective $n$-folds ($n\geq 2$) with ample canonical class.
Vakil writes that his results suggest that ``unless there is some natural reason for the [deformation] space to be
well-behaved, it will be arbitrarily badly behaved.''

On the other hand, toric varieties are so special that there may well be a natural reason for the deformation space to be well-behaved.
In fact, Murphy's Law is false for smooth toric varieties! This follows e.g.~from \cite[Theorem 6.5]{ilten2}, which implies in particular that the versal deformation space of a smooth complete toric variety cannot be a fat point.

This means that the deformation theory of smooth complete toric varieties may belong to the small class of deformation problems which are obstructed, yet one can still hope to completely describe in some explicit manner.
The next natural question to address is:
\begin{question}
Is the versal deformation of a smooth complete toric variety cut out by quadrics?
\end{question}
\noindent In fact, if we knew that the versal deformation was cut out by quadrics, then our results here would completely determine those equations. At the moment, we have far too little evidence to posit an answer one way or the other.

\vspace{.5cm}

The remainder of this paper is organized as follows. In \S \ref{sec:prelim}, we recall basic facts of \v{C}ech cohomology and toric geometry. In \S \ref{sec:obs}, we prove Proposition \ref{prop:obs}, describing $T^2_X$ combinatorially. The main work of this paper is contained in \S \ref{sec:cup}, where we prove our combinatorial description of the cup product (Theorem \ref{thm:main}). In \S\ref{sec:cycle} we show how the cup product pulls back to simple cycles $\alpha$. Finally, in \S \ref{sec:ex}, we present an example of an obstructed smooth toric threefold, proving Corollary \ref{cor:obs}.

\section{Preliminaries}\label{sec:prelim}
\subsection{\v{C}ech Cohomology}\label{sec:cech}
We begin by recalling basics of \v{C}ech cohomology and fixing notation. See e.g.~\cite[\S7.6]{bosch} for more details. Let $X$ be a topological space and $\U=\{U_i\}_{i\in I}$ either an open or closed cover of $X$. For any sheaf $\F$ of abelian groups on $X$, the group 
	$\check{C}^p_\sing(\U,\F)$
of singular $p$th \v{C}ech cochains is
\[
\check{C}^p_\sing(\U,\F)
=	\prod_{(i_0,\ldots, i_p)\in I^{p+1}}\F(U_{i_0}\cap\ldots\cap U_{i_p}).
\]
The differential 
$d^p:\check{C}^p_\sing(\U,\F)\to \check{C}^p_\sing(\U,\F)$
is defined by $d^p(f)=g$, where 
\[
	g_{i_0\ldots i_{p+1}}=\sum_{k=0}^{p+1}(-1)^k (f_{i_0\ldots \widehat{i_k} \ldots i_{p+1}})_{|U_{i_0}\cap\ldots\cap U_{i_{p+1}}}.
\]
The $p$th singular \v{C}ech cohomology group of $\F$ with respect to the cover $\U$ is the $p$th cohomology $\ch_\sing^p(\U,\F)$ of the complex $(\check{C}^\bullet_\sing,d^\bullet)$.
Elements of the kernel of $d^p$ are called singular \v{C}ech cocycles.

It is more common to work with either \emph{alternating} or \emph{ordered} \v{C}ech cohomology, since these have bounded length and involve fewer terms. We will opt to consistently work with alternating \v{C}ech cohomology: if we do not explicitly specify that we are talking about \emph{singular} \v{C}ech cohomology, then we are referring to alternating \v{C}ech cohomology. This is defined as follows.

The group of (alternating) $p$th \v{C}ech cochains $\check{C}^p(\U,\F)$ is the subgroup of $\check{C}_\sing^p(\U,\F)$ consisting of elements $f$ satisfying
\[
	f_{i_{\pi(0)}\ldots i_{\pi(p)}}=\sign(\pi)f_{i_0\ldots i_p}
\]
for any permutation $\pi\in S_{p+1}$ of $0,\ldots,p$, and 
\[
	f_{i_0\ldots i_p}=0
\]
if any index $i_j$ is repeated. After eliminating terms with doubled indices, the differential $d^p$ on the singular \v{C}ech complex also gives a differential on the subcomplex $\check{C}^\bullet(\U,\F)$.
The $p$th (alternating) \v{C}ech cohomology group of $\F$ with respect to $\U$ is the $p$th cohomology $\ch^p(\U,\F)$ of this subcomplex. Elements of $\check{C}^p(\U,\F)$ in the kernel of $d^p$ are called (alternating) \v{C}ech cocycles.

The inclusion of complexes $\check{C}^\bullet(\U,\F)\hookrightarrow \check{C}_\sing^\bullet(\U,\F)$ induces homomorphisms of cohomology groups $\ch^p(\U,\F)\to \ch_\sing^p(\U,\F)$. In fact, on the level of cohomology, these maps are isomorphisms, see \cite[\S7.6 Lemma 1]{bosch}. For our purposes, we need a map $\check{C}_\sing^\bullet(\U,\F)\to\check{C}^\bullet(\U,\F)$ which on cohomology induces the inverse of this isomorphism:
\begin{lemma}\label{lemma:phi}
	Assume that $\F$ is a sheaf of $\QQ$-modules.
	The maps
\[
\phi^p:\check{C}_\sing^p(\U,\F)\to\check{C}^p(\U,\F)
\]
defined by
\[
	\phi^p(f)_{i_0\ldots i_p}=\frac{1}{(p+1)!}\sum_{\pi\in S_{p+1}}\sign(\pi)f_{i_{\pi(0)}\ldots i_{\pi(p)}}
\]
give a homormophism of complexes. The induced map on cohomology is an isomorphism inverse to the map induced by the inclusion of $\check{C}^\bullet$ in $\check{C}_\sing^\bullet$.
\end{lemma}
\begin{proof}
	To show that $\phi^{\bullet}$ is a homomorphism of complexes, by linearity
it suffices to consider images of elements $f\in \ch_\sing^p(\U,\F)$ which are contained in a single summand.
The equality	$d^p(\phi(f))=\phi^{p+1}(d^p(f))$ then follows from a direct computation.

It is straightforward to check that $\phi^p$ is a section to the inclusion of $\ch^p(\U,\F)$ in $\ch_\sing^p(\U,\F)$. Since this inclusion induces an isomorphism on cohomology, it follows that $\phi^p$ does as well.
\end{proof}

\begin{rem}\label{rem:char}
If $\F$ is a sheaf of modules over a field $\KK$ of characteristic $q\neq 0$, one may still define the map $\phi$ as in Lemma \ref{lemma:phi} for those $p$ such that $p<q-1$. It follows that it will still induce an isomorphism of cohomology for $p<q-2$. In particular, since we are always assuming that our base field $\KK$ doesn't have characteristic two or three, we will always obtain isomorphisms in cohomology for $p=0,1,2$.
\end{rem}

In the following, we will be using \v{C}ech cohomology in two situations. The first is when $X$ is an algebraic variety, $\F$ is a coherent sheaf, and $\U$ is a particular open affine cover. In this case $\ch^p(\U,\F)$ is canonically isomorphic to the sheaf cohomology $H^p(X,\F)$ \cite[Theorem 4.5]{hartshorne}, so we will usually just write $H^p(X,\F)$. The second situation is when $X$ is a finite simplicial complex, $\F$ is the constant sheaf with coefficients in $\KK$, and $\U$ is a particular cover by closed simplices, all of whose intersections are contractible. In this case, $\ch^p(\U,\F)$ is canonically isomorphic to the simplicial cohomology groups $H^p(X,\KK)$ \cite[\S II.5.2]{godement}, and we will again usually just write $H^p(X,\KK)$.
\subsection{Cup products}\label{sec:cupdef}
Assume now that $\F$ is a sheaf of algebras on a topological space $X$ with covering $\U$. The multiplication in $\F$ induces a \emph{cup product} in cohomology
\[
	\ch^p(\U,\F)\times\ch^q(\U,\F)\to\ch^{p+q}(\U,\F).
\]
This is described for the \emph{singular} \v{C}ech cohomology groups as follows, see e.g.~\cite[\S7.6 Exercise 6]{bosch}. Given singular $p$- and $q$-cocycles $f=(f_{i_0\ldots i_p})$ and $f'=(f_{i_0\ldots i_q}')$, the cup product of the cohomology classes represented by $f$ and $f'$ is represented by the $p+q$ cocycle
$g=(g_{i_0\ldots i_{p+q}})$ with
\begin{equation}\label{eqn:cup}
	g_{i_0\ldots i_{p+q}}=f_{i_0\ldots i_p}*f'_{i_p\ldots i_{p+q}}
\end{equation}
where $*$ denotes the product on $\F$. This product gives $\bigoplus_p \ch_\sing^p(\U,\F)$ the structure of a graded associative algebra.

For our purposes, we desire a description similar to \eqref{eqn:cup} for the cup product between \emph{alternating} \v{C}ech cohomology groups. This may be obtained by appropriately composing the maps between $\check{C}^\bullet(\U,\F)$ and $\check{C}_\sing^\bullet(\U,\F)$  with the cup product on singular \v{C}ech cohomology.

We will do this explicitly for the case of interest to us, namely, when $\F=\T_X$ is the tangent sheaf on an algebraic variety $X$ with product induced by the Lie bracket $[,]$, and $p=q=1$:
\begin{lemma}\label{lemma:cup}
	Let $f=(f_{ij})$ and $f=(f_{ij}')$ be \v{C}ech one-cycles in $\check{C}^{1}(\U,\T_X)$. Then the image of their cohomology classes under the cup product map
\[
	\ch^1(\U,\T_X)\times\ch^1(\U,\T_X)\to\ch^2(\U,\T_X)
\]
is represented by the two-cycle $g=(g_{ijk})$ with
\[
	g_{ijk}=
	\frac{1}{6}\left(
	[f_{ij},f'_{jk}]
	+[f_{ij},f'_{ik}]
	+[f_{ik},f'_{jk}]
	-[f_{ik},f'_{ij}]
	-[f_{jk},f'_{ik}]
	-[f_{jk},f'_{ij}]
	\right).\\
\]
\end{lemma}
\begin{proof}
To compute the cup product, we first include $f$ and $f'$ in the group of singular \v{C}ech cochains $\check{C}^1_\sing(\U,\T_X)$, and then apply \eqref{eqn:cup} to find a representative $\tilde{g}$ of the cup product as a singular two-cycle.
We obtain
\[
	\tilde{g}_{ijk}=[f_{ij},f'_{jk}].
\]
The claim now follows from Lemma \ref{lemma:phi} and a straightforward computation by setting $g=\phi^2(\tilde{g})$.
\end{proof}
\begin{rem}
	Choosing a section to the inclusion of $\check{C}^\bullet(\U,\T_X)\hookrightarrow \check{C}_\sing^\bullet(\U,\T_X)$ that is different from our preferred section $\phi$ would lead to a representation of the cup product on the cocycle level that is different from that of Lemma \ref{lemma:cup}. Our choice of section $\phi$ was motivated by the symmetry of the expression for $g_{ijk}$ in this lemma.
\end{rem}
\subsection{Toric Varieties}\label{sec:toric}
We now fix notation and review some basic facts from toric geometry. See \cite{fulton} or \cite{cls} for a more thorough introduction. Throughout the paper we will fix a lattice $M$ which is the character lattice of the algebraic torus $T=\spec \KK[M]$. The lattice $N=\Hom(M,\ZZ)$ is the lattice of one-parameter subgroups of $T$.
 We denote the $\QQ$-vector spaces associated to $M,N$ by $M_\QQ$ and $N_\QQ$.

 Given a fan $\Sigma$ in $N_\QQ$, we associate a toric variety $X_\Sigma$, see \cite[\S3.1]{cls}. The variety $X_\Sigma$ is covered by open affine varieties $U_\sigma$ as $\sigma$ ranges over maximal cones in the fan $\Sigma$, where 
 \[
	 U_\sigma=\spec \KK[M\cap \sigma^\vee];\qquad \sigma^\vee=\{u\in M_{\QQ}\ |\ v(u)\geq 0\ \forall v\in\sigma\}.
 \]
 We denote the regular function on $U_\sigma$ associated to $u\in M\cap\sigma^\vee$ by $\chi^u$.

 Important geometric properties of $X_\Sigma$ can be translated into properties on $\Sigma$. For example, the variety $X_\Sigma$ is complete if and only if the fan $\Sigma$ is complete, that is, the union of all cones in $\Sigma$ is all of $N_\QQ$ \cite[Theorem 3.4.6]{cls}.
Likewise, the variety $X_\Sigma$ is smooth if and only if $\Sigma$ is smooth, that is, each maximal $\sigma\in \Sigma$ has rays whose primitive generators are a subset of a lattice basis of $N$ \cite[Theorem 3.1.19]{cls}.
Slightly more generally, the variety $X$ is $\QQ$-factorial if and only if $\Sigma$ is a simplicial fan, that is, each maximal $\sigma\in \Sigma$ has rays whose primitive generators are linearly independent \cite[Proposition 4.2.7]{cls}.
 We will henceforth always assume that $\Sigma$ is complete and simplicial. In other words, we will assume that $X$ is $\QQ$-factorial and complete.

 We denote the rays of $\Sigma$ by $\Sigma^{(1)}$; to any ray $\rho\in\Sigma^{(1)}$ and $u\in M$ we denote by $\rho(u)$ the evaluation of the primitive lattice generator of $\rho$ at $u$. Prime torus invariant divisors of $X_\Sigma$ are in bijection with rays in $\Sigma^{(1)}$ \cite[\S4.1]{cls}. We denote the divisor corresponding to $\rho$ by $D_\rho$.
Any torus invariant divisor $D$ may be written uniquely as a sum
\[
D=	\sum_{\rho\in\Sigma^{(1)}} a_\rho D_\rho.
\]
The sheaf $\CO(D)$ has the following local description:
the function $\chi^u$ is in $H^0(U_\sigma,\CO(D))$
if and only if 
\[
\rho(u)+a_\rho\geq 0
\]
for all $\rho\in\sigma\cap\Sigma^{(1)}$.
In particular, fixing a ray $\rho$, $\chi^u\in H^0(U_\sigma,\CO(D_\rho))$ if and only if for all $\epsilon\in\sigma\cap\Sigma^{(1)}$,
\begin{equation}\label{eqn:div}
	\epsilon(u)\geq\begin{cases}
		0 & \epsilon\neq \rho,\\
		-1 & \epsilon=\rho.
	\end{cases}
\end{equation}
\subsection{The Euler Sequence}
The fundamental tool for understanding the tangent bundle on a smooth toric variety is the \emph{Euler sequence}. For $X=X_\Sigma$ complete and $\QQ$-factorial, there is an exact sequence of sheaves
\[
\xymatrix{
	0\ar[r] & N^1\otimes \CO_X \ar[r] &\bigoplus_{\rho\in\Sigma^{(1)}} \CO_X(D_\rho) \ar_{\quad\qquad\eta}[r]& \T_X \ar[r] & 0
}
\]
where $N^1$ is a finite dimensional vector space, see \cite[Theorem 8.1.6]{cls} (and dualize). This generalizes the standard Euler sequence on projective space.
We will need an explicit description of the map $\eta$. 
Following through the construction in loc.~cit.~and dualizing, one obtains that 
\begin{equation}\label{eqn:euler}
	\eta(\chi^u)=\partial(\rho,u)
\end{equation}
for $\chi^u$ a local section of $\CO(D_\rho)$, where the derivation $\partial(\rho,u)$ is defined via
\[
	\partial(\rho,u)(\chi^v)=\rho(v)\chi^{u+v}
\]
for any $v\in M$.

We will be interested in the cohomology groups of $\T_X$. The following was first observed by Jaczewski in the smooth case:
\begin{lemma}\cite{jac}\label{lemma:tcohom}
	For $p\geq 1$, the map $\eta$ induces isomorphisms
	\[
		\bigoplus_{\rho\in\Sigma^{(1)}}H^p(X,\CO_X(D_\rho))\to H^p(X,\T_X).
	\]
\end{lemma}
\begin{proof}
	This follows directly from the Euler sequence, the long exact sequence of cohomology, and the vanishing of $H^p(X,\CO_X)$ for $p\geq 1$, see \cite[Theorem 9.2.3]{cls}.
\end{proof}
\subsection{Cohomology of Divisorial Sheaves on Toric Varieties}\label{sec:cohom}
In order to understand the cohomology of $\T_X$, Lemma \ref{lemma:tcohom} implies that it will be useful to have a combinatorial description of the cohomology groups of the sheaves $\CO(D_\rho)$.
Since $T$ acts on $X=X_\Sigma$, it will also act on the spaces of sections of $\T_X$ and $\CO(D)$ for any torus invariant divisor $D$. This induces an $M$-grading on the respective cohomology groups. 
We will follow \cite[\S9.1]{cls} to describe the graded pieces of these cohomology groups. We go into what might seem more detail than necessary since we will later need explicit descriptions of the maps between various isomorphic cohomology groups.

Let $D=\sum a_\rho D_\rho$ be any torus invariant divisor. Fixing some $u\in M$, we define the simplicial complex
\[
	V_{D,u}=\bigcup_{\sigma\in \Sigma} \conv\{n_\rho\ |\ \rho\in \sigma\cap\Sigma^{(1)}\ \textrm{and}\ \rho(u)+a_\rho<0\}\subset N_\QQ
\]
where $n_\rho$ is the primitive generator of any ray $\rho$.
For each $\sigma\in\Sigma$, there is a natural exact sequence 
\[
	0\to H^0(U_\sigma,\CO_X(D))_u\to \KK \to H^0(V_{D,u}\cap\sigma,\KK)\to 0,
\]
see \cite[Equation 9.1.10]{cls}.
Here, 
	$H^0(U_\sigma,\CO_X(D))_u$
denote the degree $u$ piece of $H^0(U_\sigma,\CO_X(D))$.
	Let $I$ be the set of maximal cones in $\Sigma$; we consider the open cover $\U=\{U_\sigma\}_{\sigma\in I}$ of $X_\Sigma$. Likewise, we have a closed cover $\V=\{V_\sigma\}_{\sigma \in I}$ of $V_{D,u}$, where $V_\sigma=V_{D,u}\cap \sigma$. The above exact sequence thus leads to an exact sequence
\begin{equation}\label{eqn:complex}
	0\to \check{C}^p(\U,\CO(D))_u\to\check{C}^p(\W,\KK)\to \check{C}^p(\V,\KK)\to 0
\end{equation}
where $\W=\{W_\sigma\}_{\sigma \in I}$ is the trivial closed cover of a single point $x$ with each $W_\sigma=x$. This sequence is compatible with the \v{C}ech differentials, so we obtain an exact sequence of \v{C}ech complexes.
Since $H^0(x,\KK)=\KK$ and $H^p(x,\KK)=0$ for $p>0$, the long exact sequence of cohomology implies that the connecting homomorphisms
\begin{equation}\label{eqn:scohom}
	\ch^{p-1}(\V,\KK)\to\ch^{p}(\U,\CO(D))_u
\end{equation}
are isomorphisms if $p\geq 2$, and for $p=1$ we have the exact sequence
\[
	0\to H^0(X,\CO(D))_u\to \KK\to H^0(V_{D,u},\KK)\to H^1(X,\CO(D))_u\to 0.
\]
This final exact sequence induces an isomorphism between the reduced cohomology $\widetilde{H}^0(V_{D,u},\KK)$ and $H^1(X,\CO(D))_u$, see \cite[Theorem 9.1.3]{cls}.

\section{Tangent and Obstruction Spaces}\label{sec:obs}
As before, we are considering a  complete $\QQ$-factorial toric variety $X=X_{\Sigma}$. 
For $\rho\in \Sigma^{(1)}$ and $u\in M$, we define
\[
	V_{\rho,u}:=V_{D_\rho,u}
\]
and notice that the vertices of $V_{\rho,u}$ have the following concrete description: for $\epsilon\in\Sigma^{(1)}$, $n_\epsilon\in V_{\rho,u}$ if and only if
\begin{enumerate}
	\item $\epsilon=\rho$ and $\epsilon(u)<-1$; or
	\item $\epsilon\neq \rho$ and $\epsilon(u)<0$.
\end{enumerate}
We define $\Gamma_{\rho,u}$ and $K_{\rho,u}$ to respectively be the one- and two-skeleta of $V_{\rho,u}$. More generally, let $V_{\rho,u}^{(p)}$ denote the $p$-skeleton of $V_{\rho,u}$.
Below we will come to see that we only need to consider the special case when $\rho(u)=-1$, in which case the description of the vertices of $V_{\rho,u}$ simplifies and $n_\rho$ itself is never a vertex of $V_{\rho,u}$.

We briefly comment on the decomposition
\begin{equation*}
H^1(X,\T_X)=\bigoplus_{\rho\in\Sigma^{(1)}} \bigoplus_{\substack{u\in M\\ \rho(u) =-1}} \widetilde{H}^0(\Gamma_{\rho,u},\KK).
\end{equation*}
This was shown in \cite{ilten1} in the smooth case; it was noted in \cite{mavlyutov2} that this also holds in the $\QQ$-factorial case. The decomposition arises by combining Lemma \ref{lemma:tcohom}
with the isomorphism 
between $\widetilde{H}^0(V_{D_\rho,u},\KK)$ and $H^1(X,\CO(D_\rho))$ described in \S\ref{sec:cohom}. One then notes that
$\widetilde{H}^0(V_{D_\rho,u},\KK)=\widetilde{H}^0(\Gamma_{\rho,u},\KK)$, and this is non-zero only if $\rho(u)=-1$.

A similar argument to the one above yields a description of $H^p(X,\T_X)$ for all $p\geq 1$:
\begin{prop}\label{prop:obs2}
For $p\geq 1$, the space $H^p(X,\T_X)$ may be decomposed as
	\begin{equation*}
H^p(X,\T_X)=\bigoplus_{\substack{(\rho,u)\in \Sigma^{(1)}\times M\\ \rho(u) =-1}} \widetilde{H}^{p-1}(V_{\rho,u},\KK)
=\bigoplus_{\substack{(\rho,u)\in\Sigma^{(1)}\times M\\ \rho(u) =-1}} \widetilde{H}^{p-1}(V_{\rho,u}^{(p)},\KK).
\end{equation*}
In particular, the space $H^2(X,\T_X)$ may be decomposed as
	\begin{equation*}
H^2(X,\T_X)=\bigoplus_{\rho\in\Sigma^{(1)}} \bigoplus_{\substack{u\in M\\ \rho(u) =-1}} H^1(K_{\rho,u},\KK).
\end{equation*}
\end{prop}
\begin{proof}
	By Lemma \ref{lemma:tcohom}, we have an $M$-graded isomorphism
	\[
		\bigoplus_{\rho\in\Sigma^{(1)}} H^p(X,\CO_X(D_\rho)) \to H^p(X,\T_X).
	\]
	Coupled with equation \eqref{eqn:scohom}, we obtain
	\[
		H^p(X,\T_X)_u\cong \bigoplus_{\rho\in\Sigma^{(1)}} \widetilde{H}^{p-1}(V_{\rho,u},\KK).
	\]

	We now show that $\widetilde{H}^{p-1}(V_{\rho,u},\KK)=0$ unless $\rho(u)=-1$. From the explicit description of $V_{\rho,u}$ above,
we observe that if $\rho(u)\neq -1$, then $V_{\rho,\lambda\cdot u}$ is the same for any $\lambda\in \ZZ_{>0}$. 
In particular, if $\widetilde{H}^{p-1}(V_{\rho,u},\KK)\neq 0$, $H^p(X,\CO(D_\rho))$ would be an infinite dimensional $\KK$-vector space, which is impossible since $X$ is complete.
We conclude that we must only consider those pairs $(\rho,u)$ such that $\rho(u)=-1$.

Finally, the $(p-1)$st reduced cohomology of $V_{\rho,u}$ is the same as that of its $p$-skeleton $V_{\rho,u}^{(p)}$.
\end{proof}

We will be interested in special zero-cocycles $f=(f_\sigma)$
representing elements of $\widetilde{H}^0(\Gamma_{\rho,u},\KK)$ coming from a connected component $Z$ of $\Gamma_{\rho,u}$. For such a connected component $Z$, we define $f(Z)=(f(Z)_{\sigma})$ by
\begin{equation}\label{eqn:comp}
	f(Z)_\sigma=\begin{cases}
1 & \sigma\cap Z\neq \emptyset\\
0 & \sigma\cap Z= \emptyset
	\end{cases}.
\end{equation}
These will be useful cocycles for us, since the classes of $\{f(Z)\}$ form a basis for $H^0(\Gamma_{\rho,u},\KK)$ as $Z$ ranges over all connected components of $\Gamma_{\rho,u}$. In particular, they provide a spanning set for
$\widetilde{H}^0(\Gamma_{\rho,u},\KK)$.
If we are instead considering a connected component $Z'$ of $\Gamma_{\rho',u'}$, we will use the notation ${f'(Z')}$.

\section{Combinatorial Description of Cup Product}\label{sec:cup}
\subsection{Mapping to $H^2(X,\T_X)$}
Fix $\rho,\rho'\in\Sigma^{(1)}$ and $u,u'\in M$ satisfying $\rho(u)=\rho'(u')=-1$.
We now describe the map
\[
	\widetilde{H}^0(\Gamma_{\rho,u},\KK)\times\widetilde{H}^0(\Gamma_{\rho',u'},\KK)\to H^2(X,\T_X)_{u+u'}
\]
induced by the cup product in terms of \v{C}ech cocycles:
\begin{lemma}\label{lem:maptoh2}
	Let $f=(f_\sigma),f'=(f_\sigma')$ be \v{C}ech zero-cocycles of $\Gamma_{\rho,u}$ and $\Gamma_{\rho',u'}$. The image in $H^2(X,\T_X)$ of the corresponding reduced cohomology classes under the cup product
	is represented by the \v{C}ech two-cycle $\theta=(\theta_{\sigma\tau\gamma})$, where
\begin{align*}
	\theta_{\sigma\tau\gamma}	=\frac{1}{2}\Big(
	f_{\sigma}f_{\tau}'-f_{\tau}f_{\sigma}'
	+f_{\gamma}f_{\sigma}'-f_{\sigma}f_{\gamma}'
	+f_{\tau}f_{\gamma}'-f_{\gamma}f_{\tau}'
	\Big)\\
	\cdot 	\Big(\rho(u')\partial(\rho',u+u')-\rho'(u)\partial(\rho,u+u')\Big).
\end{align*}
\end{lemma}
\begin{proof}
	We just need to trace through the inclusions of $\widetilde{H}^0(\Gamma_{\rho,u},\KK)$ and $\widetilde{H}^0(\Gamma_{\rho',u'},\KK)$ in $H^1(X,\T_X)$
	and compose with the description of the cup product found in Lemma \ref{lemma:cup}.
First, mapping $f$ to a cohomology class in $H^1(X,\CO(D_\rho))_u$, 
we must use the first connecting homomorphism of \eqref{eqn:complex}. We do this by sending $f$ to $a=(a_\sigma)\in \check{C}^0(\W,\KK)$ with $a_\sigma=f_\sigma$ and applying the differential $d^0$ to obtain $b=d^0(a)\in\check{C}^1(\W,\KK)$ with
\[
	b_{\sigma\tau}=a_\tau-a_\sigma=f_\tau-f_\sigma.
\]
By construction, this is the image of the element $c\in \check{C}^1(\U,\CO(D_\rho))_u$ where
\[
	c_{\sigma\tau}=(f_\tau-f_\sigma)\cdot\chi^u.
\]
Mapping further to $H^1(X,\T_X)_u$ using Lemma \ref{lemma:tcohom}, we obtain the cocycle $g=(g_{\sigma\tau})$, where
\[
	g_{\sigma\tau}=(f_\tau-f_\sigma)\partial(\rho,u).
\]
A similar computation holds for $f'$; we denote the corresponding one-cocycle in $\check{C}^1(\U,\T_X)_{u'}$ by $g'$.

Before applying Lemma \ref{lemma:cup}, we note the straightforward calculation
\begin{align*}
	[\partial(\rho,u),\partial(\rho',u')]=
	\partial(\rho,u)\circ\partial(\rho',u')-
	\partial(\rho',u')\circ\partial(\rho,u)\\
	=\rho(u')\partial(\rho',u+u')-\rho'(u)\partial(\rho,u+u').
\end{align*}
Taking this into account while applying the lemma to $g$ and $g'$, we obtain 
the two-cocycle $\theta$ with
\begin{align*}
	\theta_{\sigma\tau\gamma}=
	\frac{1}{6}\Big(
(f_\tau-f_\sigma)(f'_\gamma-f'_\tau)
+(f_\tau-f_\sigma)(f'_\gamma-f'_\sigma)
+(f_\gamma-f_\sigma)(f'_\gamma-f'_\tau)\\
-(f_\gamma-f_\sigma)(f'_\tau-f'_\sigma)
-(f_\gamma-f_\tau)(f'_\gamma-f'_\sigma)
-(f_\gamma-f_\tau)(f'_\tau-f'_\sigma)
	\Big)\\
	\cdot 	\Big(\rho(u')\partial(\rho',u+u')-\rho'(u)\partial(\rho,u+u')\Big).
\end{align*}
This simplifies to the expression in the claim.
\end{proof}
\subsection{Lifting to $H^2(X,\CO(D_\rho))$}
We now show how to lift the cocycle $\theta$ of Lemma \ref{lem:maptoh2} to a cocycle representing an element of 
\[H^2(X,\CO(D_\rho))_{u+u'}\oplus H^2(X,\CO(D_{\rho'}))_{u+u'}.\]
\begin{lemma}\label{lem:maptodh2}
	Assume that $\rho\neq \rho'$.
	With $f,f',\theta$ as in Lemma \ref{lem:maptoh2}, 
define 
	\begin{align*}
		&\kappa=(\kappa_{\sigma\tau\gamma})\qquad \kappa'=(\kappa_{\sigma\tau\gamma}');\\
	&\kappa_{\sigma\tau\gamma}	=\frac{\rho'(u)}{2}\Big(
	f_{\sigma}f_{\tau}'-f_{\tau}f_{\sigma}'
	+f_{\gamma}f_{\sigma}'-f_{\sigma}f_{\gamma}'
	+f_{\tau}f_{\gamma}'-f_{\gamma}f_{\tau}'
	\Big)
	\cdot \chi^{u+u'};\\
	&\kappa_{\sigma\tau\gamma}'	=\frac{\rho(u')}{2}\Big(
	f_{\sigma}f_{\tau}'-f_{\tau}f_{\sigma}'
	+f_{\gamma}f_{\sigma}'-f_{\sigma}f_{\gamma}'
	+f_{\tau}f_{\gamma}'-f_{\gamma}f_{\tau}'
	\Big)
	\cdot \chi^{u+u'}.\\
	\end{align*}
	Then $\kappa$ and $\kappa'$ are two-cocyles in 
	$\check{C}^2(\U,\CO(D_\rho)$ and 
	$\check{C}^2(\U,\CO(D_{\rho'})$, and the image of $\kappa'-\kappa$ under the map to $\check{C}^2(\U,\T_X)$ induced by $\eta$ is exactly $\theta$. 
\end{lemma}
\begin{proof}
	It follows from the explicit description of $\eta$ in \eqref{eqn:euler} that the image of $\kappa'-\kappa$ is indeed $\theta$. So we just need to show that $\kappa$ and $\kappa'$ are two-cocycles. We will show below that for all $\sigma\tau\gamma$, $\kappa_{\sigma\tau\gamma}$ is an element of $H^0(U_{\sigma\cap\tau\cap\gamma},\CO(D_\rho))$. A similar statement will also hold for $\kappa'$. It then remains to show that $\kappa$ and $\kappa'$ are in the kernel of the differential $d^2$. But since for $\rho\neq \rho'$, $\partial(\rho,u+u')$ and $\partial(\rho',u+u')$ are linearly independent over $\KK$, the images of $\kappa$ and $\kappa'$ under the map induced by $\eta$ must lie in the kernel of the differential. It follows that $\kappa$ and $\kappa'$ must as well.

	So we are left to show the claim that $\kappa_{\sigma\tau\gamma}$ is an element of $H^0(U_{\sigma\cap\tau\cap\gamma},\CO(D_\rho))$ for arbitrary choice of $\sigma,\tau,\gamma$. By bilinearity of the cup product, it suffices to do this for the special cases when $f=f(Z)$ and $f'=f'(Z')$ as in \eqref{eqn:comp} for connected components $Z,Z'$ of $\Gamma_{\rho,u}$ and $\Gamma_{\rho',u'}$.
In the following, we shall fix such components $Z,Z'$.

To show that $\kappa_{\sigma\tau\gamma}$ is a regular section as desired, we need to show that either $\chi^{u+u'}\in H^0(U_{\sigma\cap\tau\cap\gamma},\CO(D_\rho))$, $\rho'(u)=0$, or
\begin{align}\label{eq:zz}
	\begin{split}&f_{\sigma}f_{\tau}'-f_{\tau}f_{\sigma}'
	+f_{\gamma}f_{\sigma}'-f_{\sigma}f_{\gamma}'
	+f_{\tau}f_{\gamma}'-f_{\gamma}f_{\tau}'\\
&=f_{\sigma}(f_\tau'-f_\gamma')+f_\tau(f_\gamma'-f_\sigma')+f_\gamma(f_\sigma'-f_\tau')=0.
\end{split}
\end{align}
Let us thus assume that neither $\rho'(u)$ nor the expression in \eqref{eq:zz} is zero. Thus, all of $f_\sigma',f_\tau',f_\gamma'$ can not be equal, and by symmetry the same is true for $f_\sigma,f_\tau,f_\gamma$. Using that $f=f(Z)$ and $f'=f'(Z')$ and the symmetry of the expression, we may assume without loss of generality that we are in one of two cases:
\begin{enumerate}
	\item $f_\sigma=1,f_\tau=f_\gamma=0$; or
	\item $f_\sigma=f_\tau=1,f_\gamma=0$.
\end{enumerate}

In the first case, we must thus have $f_\tau'\neq f_\gamma'$; without loss of generality $f_\tau'=1$ and $f_\gamma'=0$. In other words, our connected component $Z$ intersects $\sigma$ but not $\tau$ and $\gamma$, whereas $Z'$ intersects $\tau$ but not $\gamma$.

Consider any ray $\epsilon$ of $\sigma\cap\tau\cap\gamma$. By \eqref{eqn:div} we must show that $\epsilon(u+u')\geq 0$ if $\epsilon\neq \rho$, and $\epsilon(u+u')\geq -1$ if $\epsilon=\rho$. 
Suppose $\epsilon=\rho$; then $\epsilon(u+u')=\epsilon(u')-1$. 
If $\epsilon(u')<0$, then $\epsilon$ is a vertex of $\Gamma_{\rho',u'}$ (note we are assuming $\rho\neq\rho'$). Now, $\epsilon$ is in $\tau$, and $Z'$ intersects $\tau$, so by convexity, $\epsilon\in Z'$. But $\epsilon$ is also in $\gamma$, contradicting $Z'\cap\gamma=\emptyset$. Hence, $\epsilon(u')\geq0$, implying $\epsilon(u+u')\geq -1$ as required.

Suppose instead that $\epsilon=\rho'$;
 then $\epsilon(u+u')=\rho'(u+u')=\epsilon(u)-1$.
 If $\epsilon(u)<0$, then $\epsilon$ is a vertex of $\Gamma_{\rho,u}$. Since $\sigma\cap Z\neq \emptyset$, convexity again implies that $\epsilon\in Z$, but this contradicts $Z\cap \tau=\emptyset$. We have also assumed that $\rho'(u)\neq 0$, so we conclude that $\epsilon(u)\geq 1$, and $\epsilon(u+u')\geq 0$.

 Finally, supposed that $\epsilon\neq \rho,\rho'$.
Arguing similarly to above, we cannot have $\epsilon(u)<0$, since then we would obtain $\epsilon\in Z$, contradicting $Z\cap\tau=\emptyset$. Likewise, we cannot have $\epsilon(u')<0$. We thus conclude $\epsilon(u+u')\geq 0$.
This concludes the argument for the first case.

In the second case, we notice that we cannot have $f_\tau'=f_\sigma'$, and may argue as in the first case after appropriately permuting the roles of $\sigma$, $\tau$, and $\gamma$. 
\end{proof}
\subsection{Lifting to simplicial cohomology}
We are now able to come to our main result:
\begin{thm}\label{thm:main2}
	Fix $\rho,\rho'\in\Sigma^{(1)}$ and $u,u'\in M$ with $\rho(u)=\rho'(u')=-1$.
	\begin{enumerate}
		\item\label{item:one} The image of
\[\widetilde{H}^0(\Gamma_{\rho,u},\KK)\times \widetilde{H}^0(\Gamma_{\rho',u'},\KK)\]
in $H^2(X,\T_X)$ under the cup product via \eqref{eqn:t1} is $0$ unless $\rho(u')=0$ or $\rho'(u)=0$.
\item\label{item:two} Assume that $\rho(u')=0$, and let
	$f=(f_\sigma)$ and $f'=(f_\sigma')$ be \v{C}ech zero-cycles of $\Gamma_{\rho,u}$ and $\Gamma_{\rho',u'}$. 
	Then the cup product of $(\overline f,\overline f')$ is contained in $H^1(K_{\rho,u+u'})$ via \eqref{eqn:t2} and may be represented by the \v{C}ech one-cocycle $g=(g_{\sigma\tau})$ where
	\[
		g_{\sigma\tau}=\begin{cases}
			\frac{\rho'(u)}{2}(f_{\sigma}f_{\tau}'-f_{\tau}f_{\sigma}') & K_{\rho,u+u'}\cap\sigma\cap\tau\neq\emptyset,\\
			0& \textrm{otherwise}.
	\end{cases}
	\]
A similar formula holds when $\rho'(u)=0$.
\end{enumerate}
\end{thm}
\begin{proof}
	Let $f,f'$ be \v{C}ech zero-cocyles of $\Gamma_{\rho,u}$ and $\Gamma_{\rho',u'}$.
	We first show item \ref{item:one}. If $\rho=\rho'$, then Lemma \ref{lem:maptoh2} implies that the image of the cup product of the classes of $f$ and $f'$ in $H^2(X,\T_X)$ is zero. So we may henceforth assume that $\rho\neq\rho'$. By Lemma \ref{lem:maptodh2}, the cup product of the classes of $f$ and $f'$ may be represented by a cocyle in
	\[
		\bigoplus_{\rho\in\Sigma^{(1)}} H^2(X,\CO(D_\rho))_{u+u'}
	\]
	living entirely in the $\rho$ and $\rho'$ summands. But it follows from Proposition \ref{prop:obs2} that these vanish respectively unless $\rho(u+u')=-1$ or $\rho'(u+u')=-1$. Given $\rho(u)=\rho'(u')=-1$, this is equivalent to $\rho(u')=0$ or $\rho'(u)=0$. This completes the proof of the first item.

	For item \ref{item:two}, assume now that $\rho(u')=0$. By Lemma \ref{lem:maptodh2}, it follows that the cup product of the classes of $f$ and $f'$ may be thought of as a class in $H^2(X,\CO(D_\rho))_{u+u'}$, represented by the cocycle $-\kappa$ with
	\begin{align*}
	&\kappa_{\sigma\tau\gamma}	=\frac{\rho'(u)}{2}\Big(
	f_{\sigma}f_{\tau}'-f_{\tau}f_{\sigma}'
	+f_{\gamma}f_{\sigma}'-f_{\sigma}f_{\gamma}'
	+f_{\tau}f_{\gamma}'-f_{\gamma}f_{\tau}'
	\Big)
	\cdot \chi^{u+u'}.
	\end{align*}

	We consider the element $\tilde{g}\in\check{C}^1(\W,\KK)$ defined by 
	\[
		\tilde{g}_{\sigma\tau}=
		\frac{\rho'(u)}{2}(f_{\sigma}f_{\tau}'-f_{\tau}f_{\sigma}').
	\]
	On the one hand, the image of $\tilde{g}$ in $\check{C}^1(\V,\KK)$ is exactly the one-cochain $g$ from the statement of the theorem.
	On the other hand, we may compute that \[
		d^1(\tilde{g})_{\sigma\tau\gamma}=
		\frac{\rho'(u)}{2}(f_\tau f_\sigma'-f_\sigma f_\tau'+f_\sigma f_\gamma'-f_\gamma f_\sigma'+f_\gamma f_\tau'- f_\tau f_\gamma')
	\]
	in $\check{C}^2(\W,\KK)$.
	This is exactly the image of $-\kappa$ under the inclusion \[\check{C}^2(\U,\CO(D_\rho))_{u+u'}\hookrightarrow \check{C}^2(\W,\KK).\]

	The exact sequence \eqref{eqn:complex} and its compatibility with the \v{C}ech differentials implies that $g$ is a cocycle in $\check{C}^1(\V,\KK)$, and that the image of its cohomology class under the connecting homomorphism is represented by $-\kappa$. This completes the proof of the second claim.
\end{proof}

\begin{rem}
With notation as in Theorem \ref{thm:main2}, one might wonder what happens to the cup product when \emph{both} $\rho(u')=0$ and $\rho'(u)=0$. It follows directly from the second part of the theorem that the image of this part of the cup product vanishes.
\end{rem}

\section{Pulling back to cycles}\label{sec:cycle}
\subsection{Setup}
Let $\alpha$ be a simple cycle in $K_{\rho,u+u'}$, that is, an oriented connected subgraph of the edges of $K_{\rho,u+u'}$ in which no edges are repeated and every vertex has degree $2$. Such a cycle $\alpha$ gives rise to a one-cycle $[\alpha]$ in the simplicial homology $H_1(K_{\rho,u+u'},\KK)$ by considering the sum
\[
	[\alpha]=\sum_{E\in \alpha} \pm E
\]
with signs depending on the orientation of $\alpha$ and the chosen orientation of $K_{\rho,u+u'}$.
This similarly determines a distinguished generator of $H_1(\alpha,\KK)$ which we will also denote by $[\alpha]$.
We denote the element in $H^1(\alpha,\KK)$ dual to the class in $H_1(\alpha,\KK)$ of $[\alpha]$
by $\alpha_\fun$.

\begin{defn}
	A \emph{$\Sigma$-reduced cycle} is a simple cycle  $\alpha$  in $K_{\rho,u+u'}$, with $[\alpha]$ not  homologous to zero, such that no edges of $\alpha$ are contained in a common cone of $\Sigma$.
\end{defn}
By the following lemma, it will suffice to consider only $\Sigma$-reduced cycles:
\begin{lemma}
	Any class in $H_1(K_{\rho,u+u'},\KK)$ may be written as a sum of classes of $\Sigma$-reduced cycles.
\end{lemma}
\begin{proof}
	It suffices to show the lemma for classes represented by $[\beta]$, for some simple cycle $\beta$. Suppose that $\beta$ is not $\Sigma$-reduced, with edges $E_1$ and $E_2$ contained in a cone $\sigma$. If $E_1=[v_1,v_2]$ and $E_2=[v_2,v_3]$, then we may replace these edges in $\beta$ by $E=[v_1,v_3]$ and obtain an equivalent homology class. The resulting simple cycle $\beta'$ has one fewer edge.
	
	Similarly, if $E_1\cap E_2=\emptyset$, we may split $\beta$ into simple cycles $\beta'$, $\beta''$ such that $[\beta]$ is homologous to $[\beta']+[\beta'']$ and both these cycles have fewer edges.

	The claim now follows by infinite descent.
\end{proof}

Given a $\Sigma$-reduced cycle $\alpha$, the closed cover $\V$ of $K_{\rho,u+u'}$ induces a closed cover $\V^\alpha=\{V_\sigma^\alpha\}$ of $\alpha$:
\[
V_\sigma^\alpha=\alpha\cap\sigma.
\]
This closed cover is contractible, so $H^p(\V^\alpha,\KK)=H^p(\alpha,\KK)$.

There is a natural map of complexes 
\[
	\check{C}^\bullet(\V,\KK)\to
	\check{C}^\bullet(\V^\alpha,\KK)
\]
induced by restriction. The corresponding map of cohomology 
\[
	\iota_\alpha^*:H^p(K_{\rho,u+u'},\KK)\to H^p(\alpha,\KK)
\]
is the pullback morphism corresponding to the inclusion $\iota_\alpha:\alpha\hookrightarrow K_{\rho,u+u'}$.
In particular, for any $\omega\in H^1(K_{\rho,u+u'},\KK)$, 

\[
	\iota_\alpha^*(\omega)=\langle \alpha,\omega\rangle \cdot \alpha_\fun
\]
where $\langle ~,~\rangle$ denotes the pairing between $H_1$ and $H^1$.

\subsection{Computing the cup product}\label{sec:cyclethm}
Fix a $\Sigma$-reduced cycle $\alpha$, and let $Z$ and $Z'$ be connected components of $\Gamma_{\rho,u}$ and $\Gamma_{\rho',u'}$. 
As in Theorem \ref{thm:main2}, assume that $\rho(u')=0$.
We now show how to compute $\iota_\alpha^{*}(\omega)$ directly, where $\omega$ is the class of $H^1(K_{\rho,u+u'},\KK)$ corresponding to the cup product 
 of the classes of $f(Z)$ and $f'(Z')$.

Let $E(\alpha)$ denote the set of edges $\alpha$. We will write
\[
	E(\alpha)=\{E_1,\ldots,E_k\}	
\]
with the edges ordered cyclically modulo $k$ ($E_{i+1}$ is the edge following $E_i$).
For each edge $E_i$ of $\alpha$ we fix a maximal cone $\sigma_i$ with $\sigma_i\cap\alpha=E_i$. These cones exist since $\alpha$ is $\Sigma$-reduced.

To $Z$ we associate a subset $\alpha(Z)$ of $E(\alpha)$:
\begin{align*}
	\alpha(Z)=\{E_i\in E(\alpha)\ |\ Z\cap \sigma_i\neq \emptyset\}.
\end{align*}
The set $\alpha(Z')$ is defined analogously.
An index $1\leq i \leq k$ is \emph{relevant} if 
\begin{align*}
	&\alpha(Z)\cap\{E_i,E_{i+1}\},\quad\alpha(Z')\cap\{E_i,E_{i+1}\}
\end{align*}
are not equal but both non-empty, and
their union contains both $E_i$ and $E_{i+1}$. Here, indices are taken modulo $k$.

We set
\begin{align*}
Z*_\alpha Z':&=
\frac{\rho'(u)}{2}\cdot \sum_{i\ \textrm{relevant}}b_i\\
b_i&=\begin{cases}
	1 & E_i\in \alpha(Z),\ E_{i+1} \in \alpha(Z'),\\
	-1 & E_i\in \alpha(Z'),\ E_{i+1} \in \alpha(Z).\\
\end{cases}
\end{align*}
See Figure \ref{fig:bi} for an illustration of various values of $b_i$: for each relevant index $i$, the value of $b_i$ is recorded in parenthesis next to the edge $E_i$. For example in the leftmost case, there are two relevant indices, each with $b_i=1$.

\begin{figure}
	\begin{tikzpicture}[scale=.5]
		\begin{scope}[shift={(-.07,.07)}]	
			\draw [blue,dotted, very thick] (6,6) -- (4,5) -- (3,3);
			\draw [blue] node at (2.5,5) {$\alpha(Z)$};
		\end{scope}
		\begin{scope}[shift={(.07,-.07)}]	
			\draw [red,very thick] (4,5) -- (3,3) -- (3,1);
			\draw [red] node at (4.5,3) {$\alpha(Z')$};
		\end{scope}
		\draw node at (5,6.2) {$(1)$};
		\draw node at (2.5,4) {$(1)$};
		\draw [->] (6,6) -> (5,5.5);
		\draw  (5,5.5) -> (4,5);
		\draw [->] (4,5) -> (3.5,4);
		\draw  (3.5,4) -> (3,3);
		\draw [->] (3,3) -> (3,2);
		\draw  (3,2) -> (3,1);
		\draw [->] (3,1) -> (4,0);
		\draw  (4,0) -> (5,-1);
		\draw [fill=black] (6,6) circle [, radius=.08];	
		\draw [fill=black] (4,5) circle [, radius=.08];	
		\draw [fill=black] (3,3) circle [, radius=.08];	
		\draw [fill=black] (3,1) circle [, radius=.08];	
		\draw [fill=black] (5,-1) circle [, radius=.08];
		\draw [dashed] (5,-1) -- (6,-1);
		\draw [dashed] (6,6) -- (7,6);

		\begin{scope}[shift={(9,0)}]	
		\begin{scope}[shift={(-.12,.0)}]	
			\draw [blue,dotted,very thick] (6,6) -- (4,5) -- (3,3) -- (3,1) -- (5,-1);
			\draw [blue] node at (2.5,5) {$\alpha(Z)$};
		\end{scope}
		\begin{scope}[shift={(.12,0)}]	
			\draw [red,very thick] (4,5) -- (3,3) -- (3,1);
			\draw [red] node at (4.5,3) {$\alpha(Z')$};
		\end{scope}
		\draw node at (5,6.2) {$(1)$};
		\draw node at (2,2) {$(-1)$};
		\draw [->] (6,6) -> (5,5.5);
		\draw  (5,5.5) -> (4,5);
		\draw [->] (4,5) -> (3.5,4);
		\draw  (3.5,4) -> (3,3);
		\draw [->] (3,3) -> (3,2);
		\draw  (3,2) -> (3,1);
		\draw [->] (3,1) -> (4,0);
		\draw  (4,0) -> (5,-1);
		\draw [fill=black] (6,6) circle [, radius=.08];	
		\draw [fill=black] (4,5) circle [, radius=.08];	
		\draw [fill=black] (3,3) circle [, radius=.08];	
		\draw [fill=black] (3,1) circle [, radius=.08];	
		\draw [fill=black] (5,-1) circle [, radius=.08];
		\draw [dashed] (5,-1) -- (6,-1);
		\draw [dashed] (6,6) -- (7,6);
	\end{scope}
	
		\begin{scope}[shift={(18,0)}]	
		\begin{scope}[shift={(-.12,.0)}]	
			\draw [blue,dotted, very thick] (6,6) -- (4,5) -- (3,3);
			\draw [blue] node at (2.5,5) {$\alpha(Z)$};
		\end{scope}
		\begin{scope}[shift={(.12,0)}]	
			\draw [red,very thick] (3,3) -- (3,1);
			\draw [red] node at (4.5,3) {$\alpha(Z')$};
		\end{scope}
		\draw [->] (6,6) -> (5,5.5);
		\draw  (5,5.5) -> (4,5);
		\draw [->] (4,5) -> (3.5,4);
		\draw  (3.5,4) -> (3,3);
		\draw [->] (3,3) -> (3,2);
		\draw  (3,2) -> (3,1);
		\draw [->] (3,1) -> (4,0);
		\draw  (4,0) -> (5,-1);
		\draw [fill=black] (6,6) circle [, radius=.08];	
		\draw [fill=black] (4,5) circle [, radius=.08];	
		\draw [fill=black] (3,3) circle [, radius=.08];	
		\draw [fill=black] (3,1) circle [, radius=.08];	
		\draw [fill=black] (5,-1) circle [, radius=.08];
		\draw [dashed] (5,-1) -- (6,-1);
		\draw [dashed] (6,6) -- (7,6);
		\draw node at (2.8,4) {$(1)$}; 
	\end{scope}
	\end{tikzpicture}

	\caption{Values of $b_i$}\label{fig:bi}
\end{figure}
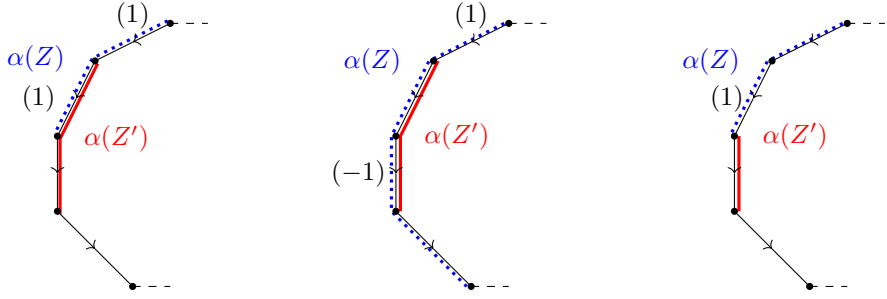

\begin{thm}\label{thm:cycle}
With the above notation,
\[
\iota_\alpha(\omega)^*=(Z*_\alpha Z')\cdot\alpha_\fun.
\]
In particular, the image of $\widetilde{H}^0(\Gamma_{\rho,u},\KK)\times\widetilde{H}^0(\Gamma_{\rho',u'}, \KK)$ under the cup product is zero if and only if for all $\Sigma$-reduced cycles $\alpha$ and all connected components $Z,Z'$ of $\Gamma_{\rho,u}$ and $\Gamma_{\rho',u'}$, $Z*_{\alpha}Z'=0$.
\end{thm}
\begin{proof}
	We know by Theorem \ref{thm:main2} that $\omega$ is represented in 
	$H^1(K_{\rho,u+u'},\KK)$ by the cocyle $g=(g_{\sigma\tau})$, where
\[		g_{\sigma\tau}=\begin{cases}
			\frac{\rho'(u)}{2}(f_{\sigma}f_{\tau}'-f_{\tau}f_{\sigma}') & K_{\rho,u+u'}\cap\sigma\cap\tau\neq\emptyset,\\
			0& \textrm{otherwise}.
	\end{cases}
\]
The image of this cocyle under the map
\[
	\check{C}^1(\V,\KK)\to
	\check{C}^1(\V^\alpha,\KK)
\]
is thus $(h_{\sigma\tau})$ where
\[		h_{\sigma\tau}=\begin{cases}
			\frac{\rho'(u)}{2}(f_{\sigma}f_{\tau}'-f_{\tau}f_{\sigma}') & \alpha\cap\sigma\cap\tau\neq\emptyset,\\
			0& \textrm{otherwise}.
	\end{cases}
\]

There is an easier closed cover $\W^{\alpha}=(E_i)$ of $\alpha$ that we would like to use; it is indexed by $1\leq i \leq k$. This is again a closed cover with all intersections contractible, so it also computes $H^1(\alpha,\KK)$.
The assignment $i\mapsto \sigma_{i}$ lets us view $\V^{\alpha}$ as a refinement of $\W^\alpha$, and induces a map of \v{C}ech complexes
\[
	\check{C}^p(\W^\alpha,\KK)\to \check{C}^p(V^\alpha,\KK)
\]
see e.g.~\cite[Exercise III.4.4]{hartshorne}.
This map has a natural section given by forgetting entries with indices not among the $\sigma_i$; both maps induce isomorphisms on cohomology.

Hence, we may represent 
$\iota_\alpha^*(\omega)$ as a \v{C}ech one-cocycle with respect to $\W^\alpha$ by $(a_{ij})$ with
\[		a_{ij}=\begin{cases}
		\frac{\rho'(u)}{2}(f_{\sigma_{i}}f_{\sigma_j}'-f_{\sigma_j}f_{\sigma_i}') & i-j\equiv \pm 1 \mod k,\\
			0& \textrm{otherwise}.
	\end{cases}
\]
On the other hand, a straightforward computation shows that 
for any 
\v{C}ech one-cocycle $c=(c_{ij})$ with respect to $\W^\alpha$, the class represented by $c$ in $H^1(\alpha,\KK)$ is 
\[
	\left(\sum_{i=1}^k c_{i(i+1)}\right)\cdot \alpha_\fun
\]
where indices are taken modulo $k$.

We thus compute 
\begin{align*}
	\sum_{i=1}^k a_{i(i+1)}
	=\sum_{i=1}^k		\frac{\rho'(u)}{2}(f_{\sigma_{i}}f_{\sigma_{i+1}}'-f_{\sigma_{i+1}}f_{\sigma_i}')\\
	=\sum_{i\ \textrm{relevant}}		\frac{\rho'(u)}{2}(f_{\sigma_{i}}f_{\sigma_{i+1}}'-f_{\sigma_{i+1}}f_{\sigma_i}')\\
	=\sum_{i\ \textrm{relevant}}		\frac{\rho'(u)}{2}
b_i=Z*_\alpha Z'.\end{align*}
This completes the proof.
\end{proof}

\section{An Obstructed Example}\label{sec:ex}
We now consider the following concrete example. Let $N=\ZZ^3$, and define rays\begin{align*}
	\rho_1= \QQ_{\geq 0} \cdot (1,0,0),\qquad
	\rho_2= \QQ_{\geq 0} \cdot (1,0,-1),\qquad
	\rho_3= \QQ_{\geq 0} \cdot (1,0,1),\\
	\rho_4= \QQ_{\geq 0} \cdot (2,-1,0),\qquad
	\rho_5= \QQ_{\geq 0} \cdot (1,-1,0),\qquad
	\rho_6= \QQ_{\geq 0} \cdot (1,1,0),\\
	\rho_7= \QQ_{\geq 0} \cdot (0,1,-1),\qquad
	\rho_8= \QQ_{\geq 0} \cdot (0,1,1),\qquad
	\rho_9= \QQ_{\geq 0} \cdot (-1,0,0).
\end{align*}
These $\rho_i$ form the rays of a smooth complete fan $\Sigma$ whose maximal cones are spanned by
\begin{align*}
\rho_1,\rho_2,\rho_4,\qquad
\rho_1,\rho_2,\rho_7,\qquad
\rho_1,\rho_3,\rho_4,\\
\rho_1,\rho_3,\rho_8,\qquad
\rho_1,\rho_6,\rho_7,\qquad
\rho_1,\rho_6,\rho_8,\\
\rho_2,\rho_4,\rho_5,\qquad
\rho_2,\rho_5,\rho_9,\qquad
\rho_2,\rho_7,\rho_9,\\
\rho_3,\rho_4,\rho_5,\qquad
\rho_3,\rho_5,\rho_9,\qquad
\rho_3,\rho_8,\rho_9,\\
\rho_6,\rho_7,\rho_9,\qquad
\rho_6,\rho_8,\rho_9.
\end{align*}
We will see using Theorem \ref{thm:cycle} that $X_\Sigma$ has non-vanishing cup-product, and hence obstructed deformations. This will show:
\begin{cor}\label{cor:ob2}
There exists a smooth complete toric threefold with obstructed deformations.
\end{cor}
\begin{figure}
	\begin{tikzpicture}[scale=1.5]
		\draw [fill=black] (0,0) circle [radius=.04];
		\draw [fill=black] (-.5,0) circle [radius=.04];
		\draw [fill=black] (0,-1) circle [radius=.04];
		\draw [fill=black] (-1,0) circle [radius=.04];
		\draw [fill=black] (0,1) circle [radius=.04];
		\draw [fill=black] (1,0) circle [radius=.04];
\draw (0,0) -- (1,0);
\draw (0,0) -- (-1,0);
\draw (0,0) -- (0,1);
\draw (0,0) -- (0,-1);
\draw (-.5,0) -- (0,-1);
\draw (-1,0) -- (0,-1);
\draw (-.5,0) -- (0,1);
\draw (-1,0) -- (0,1);
\draw [above right]node at (0,0) {\ \ $\rho_1$};
\draw [below]node at (-.5,0) {$\quad\quad\rho_4$};
\draw [right]node at (0,-1) {$\rho_2$};
\draw [below left]node at (-1,0) {$\rho_5$};
\draw [below right]node at (0,1) {$\rho_3$};
\draw [above right]node at (1,0) {$\  \ \rho_6$};
\draw (0,0) -- (2,2);
\draw (0,1) -- (1,2);
\draw (0,0) -- (2,-2);
\draw (0,-1) -- (1,-2);
\draw (1,0) -- (3,2);
\draw (1,0) -- (3,-2);
\draw (0,1) -- (0,2);
\draw (0,-1) -- (0,-2);
\draw (-1,0) -- (-2,0);
\draw (1,0) -- (3,0);
\draw [very thick,blue] (0,1) -- (-1,0) -- (0,-1) -- (-.5,0) -- (-0,1);
\draw [very thick,blue] (-1,0) -- (-.5,0) ;
\draw [fill=blue] (1,0) circle [radius=.08];
	\end{tikzpicture}
	\caption{$\Sigma\cap [u=-1]$ for $u=(-1,0,0)$ and $\Gamma_{\rho_1,u}$}\label{fig:1}
\end{figure}
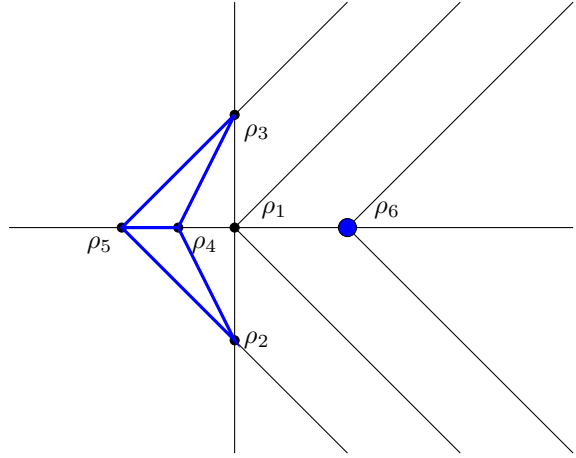

The degrees where we will look for first-order deformations are $u=(-1,0,0)$ and $u'=(0,-1,0)$. We picture the intersection of $\Sigma$ with the hyperplane 
\[
	\{v\in N_\QQ\ |\ v(u)=-1\}
\]
in Figure \ref{fig:1}. The graph $\Gamma_{\rho_1,u}$ is also pictured in the figure in blue bold lines. It has two connected components: one component contains the generators of the rays $\rho_2,\rho_3,\rho_4,\rho_5$, and the other contains the generator of $\rho_6$. We will denote the first component by $Z$.
Note that for any other choice of ray $\rho$, $\Gamma_{\rho,u}$ is connected. Hence, $H^1(X,\T_X)_u$ is one-dimensional.

\begin{figure}
	\begin{tikzpicture}[scale=1.5]
		\draw [fill=black] (1,0) circle [radius=.04];
		\draw [fill=black] (0,1) circle [radius=.04];
		\draw [fill=black] (0,-1) circle [radius=.04];
\draw [above right]node at (1,0) {\ \ $\rho_6$};
\draw [above right]node at (0,1) {\ \ $\rho_8$};
\draw [above right]node at (0,-1) {\ \ $\rho_7$};
\draw (0,1) -- (1,0) -- (0,-1);
\draw (0,-1) -- (2,-1);
\draw (0,-1) -- (1,-2);
\draw (0,1) -- (2,1);
\draw (0,1) -- (1,2);
\draw (1,0) -- (2,0);
\draw (0,-1) -- (1,-2);
\draw (0,-1) -- (-1,-1);
\draw (0,1) -- (-1,1);
\draw (1,0) -- (-1,0);
\draw [fill=blue] (0,1) circle [radius=.08];
\draw [fill=blue] (0,-1) circle [radius=.08];
	\end{tikzpicture}
	\caption{$\Sigma\cap [u'=-1]$ for $u'=(0,-1,0)$ and $\Gamma_{\rho_6,u'}$}\label{fig:2}
\end{figure}
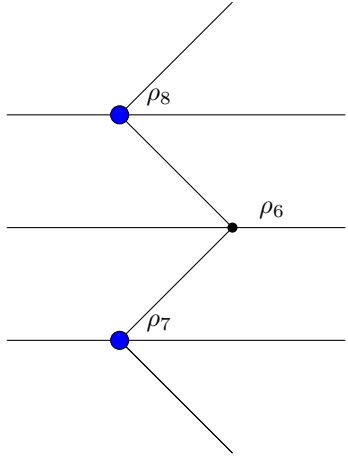

In Figure \ref{fig:2}
we picture the intersection of $\Sigma$ with the hyperplane
\[
	\{v\in N_\QQ\ |\ v(u')=-1\}
\]
along with the graph $\Gamma_{\rho_6,u'}$. This graph has two connected components, consisting of the primitive generators of $\rho_7$ and $\rho_8$. We denote the first of these components by $Z'$. Again for any other choice of ray $\rho$, $\Gamma_{\rho,u'}$ is connected, so $H^1(X,\T_X)_{u'}$ is also one-dimensional.

\begin{figure}
	\begin{tikzpicture}[scale=1.5]
\draw (-1,0) -- (.5,0) -- (1,1);
\draw (.5,0) -- (1,-1);
\draw (1,-1) -- (0,0) -- (1,1);
\draw (1,1) -- (0,1) -- (0,0) -- (0,-1) -- (1,-1);
\draw (0,1) -- (-1,0) -- (0,-1);
\draw (-1,0) -- (-2,0);
\draw (0,1) -- (-2,1);
\draw (0,-1) -- (-2,-1);
\draw (0,1) -- (0,2);
\draw (0,-1) -- (0,-2);
\draw (1,1) -- (2,2);
\draw (1,-1) -- (2,-2);
\draw (.5,0) -- (1.5,0);
\draw [line width=3.2, lightgray](.5,0) -- (1,1) -- (0,1) -- (-1,0) -- (0,-1) -- (1,-1) -- (.5,0);
\draw [line width=1.5, lightgray,->](.5,0) -- (.8,.6); 
		\draw [fill=black] (0,0) circle [radius=.04];
		\draw [fill=black] (0,-1) circle [radius=.04];
		\draw [fill=black] (0,1) circle [radius=.04];
		\draw [fill=black] (-1,0) circle [radius=.04];
		\draw [fill=black] (.5,0) circle [radius=.04];
		\draw [fill=black] (1,1) circle [radius=.04];
		\draw [fill=black] (1,-1) circle [radius=.04];
\draw [above right]node at (.5,0) {\ \ $\rho_6$};
\draw [above left]node at (0,0) {\ \ $\rho_1$};
\draw [above left]node at (0,1) {\ \ $\rho_3$};
\draw [below left]node at (0,-1) {\ \ $\rho_2$};
\draw [below left]node at (1,-1) {\ \ $\rho_7$};
\draw [above left]node at (1,1) {\ \ $\rho_8$};
\draw [above left]node at (-1,0) {\ \ $\rho_4$};
\draw [very thick,blue] (1,1.05) -- (0,1.05) -- (-1.05,0) -- (0,-1.05) -- (1,-1.05);
\draw [very thick,red] (0,-.95) -- (.95,-.95) -- (.45,0);
\draw node at (-.7,-.7) {$(1)$};
\draw node at (.4,-1.3) {$(1)$};

\end{tikzpicture}
	\caption{$\Sigma\cap [u+u'=-1]$, $\alpha$, and $\alpha(Z)$, $\alpha(Z')$}\label{fig:3}
\end{figure}
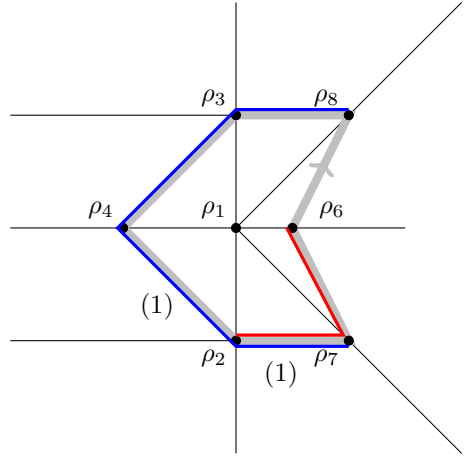

We will now compute the cup product $\omega$ of the first-order deformations corresponding to $Z$ and $Z'$. By Theorem \ref{thm:main2}, it is possible that this is non-zero, since $\rho_1(u')=0$. This class will live in degree $u+u'$, so 
we picture the intersection of $\Sigma$ with the hyperplane
\[
	\{v\in N_\QQ\ |\ v(u+u')=-1\}
\]
in Figure \ref{fig:3}.

Let $\alpha$ be the $\Sigma$-reduced cycle pictured in gray oriented in counter-clockwise direction. That is, $\alpha$ has vertices (in order) at the primitive generators of \[\rho_8,\rho_3,\rho_4,\rho_2,\rho_7,\rho_6,\rho_8.\] This is indeed $\Sigma$-reduced. To apply Theorem \ref{thm:cycle}, we must choose a maximal cone $\sigma_i\in \Sigma$ for each edge of $\alpha$. In this instance there is a canonical choice: for the edge corresponding to rays $\rho_i$ and $\rho_j$, we take the cone generated by $\rho_i,\rho_j,\rho_1$.

This gives rise to the sets $\alpha(Z)$ and $\alpha(Z')$, pictured in the figure in blue and red, respectively. The set $\alpha(Z)$ is a subgraph with vertices equal to the primitive generators of $\rho_8,\rho_3,\rho_4,\rho_2,\rho_7$, and the set $\alpha(Z')$ is a subgraph with vertices equal to the primitive generators of $\rho_2,\rho_7,\rho_6$.

The only \emph{relevant} indices are for the edges corresponding to $\rho_4,\rho_2$ and $\rho_2,\rho_7$; each contributes a value of $b_i=1$, so we obtain by Theorem \ref{thm:cycle} that
\[
\iota_\alpha^*(\omega)=1\cdot \alpha_\fun.
\]
In particular, the cup product $\omega$ of the first-order deformations corresponding to $Z$ and $Z'$ is non-zero.

\subsection*{Acknowledgements} 
Both authors were partially supported by NSERC.
The first author was partially supported by the grant 346300 for IMPAN from the Simons Foundation and the matching 2015-2019 Polish MNiSW fund.
We thank the anonymous referees for useful comments.
\bibliographystyle{alpha}
\bibliography{cup-product} 
\end{document}